\documentclass[pdflatex,sn-mathphys-num]{sn-jnl}
\usepackage{amsmath,amssymb,amsfonts,amsthm}
\usepackage{booktabs,mathtools,microtype}
\usepackage{xurl}
\providecommand{\burl}[1]{\url{#1}}
\theoremstyle{thmstyleone}
\newtheorem{theorem}{Theorem}[section]
\newtheorem{lemma}[theorem]{Lemma}
\newtheorem{proposition}[theorem]{Proposition}
\newtheorem{corollary}[theorem]{Corollary}
\theoremstyle{thmstyletwo}
\newtheorem{remark}[theorem]{Remark}
\newtheorem{example}[theorem]{Example}
\theoremstyle{thmstylethree}
\newtheorem{definition}[theorem]{Definition}
\newtheorem{assumption}[theorem]{Assumption}
\numberwithin{equation}{section}
\newcommand{\dd}{\,\mathrm{d}}
\newcommand{\R}{\mathbb R}

\newcommand{\Nzero}{\mathbb N_0}
\newcommand{\BL}{\mathrm{BL}}
\newcommand{\TV}{\mathrm{TV}}
\newcommand{\Lip}{\operatorname{Lip}}
\newcommand{\supp}{\operatorname{supp}}
\newcommand{\spanof}{\operatorname{span}}
\newcommand{\norm}[1]{\left\lVert #1\right\rVert}
\newcommand{\abs}[1]{\left\lvert #1\right\rvert}
\newcommand{\ip}[2]{\left\langle #1,#2\right\rangle}
\newcommand{\calM}{\mathcal M}
\newcommand{\calP}{\mathcal P}
\newcommand{\calJ}{\mathcal J}

\newcommand{\calK}{\mathcal K}
\newcommand{\rhoAw}{\rho_{A,w}}
\newcommand{\rhoA}{\rho_A}
\newcommand{\rhoN}{\rho_{A_N}}

\begin{document}
\hypersetup{pdftitle={Sharp spectral-scale stability for parabolic equations with measure-valued delay},pdfauthor={Lennon J. Shikhman},pdfsubject={Spectral moduli and regularity thresholds for measure-valued delay},pdfkeywords={measure-valued delay, spectral-scale stability, logarithmic regularity, parabolic equations, atomic quadrature}}
\title[Spectral-scale stability for measure-valued delay]{Sharp spectral-scale stability for parabolic equations with measure-valued delay}
\author*[1,2]{\fnm{Lennon J.} \sur{Shikhman}}\email{lshikhman2022@fit.edu}
\affil*[1]{\orgdiv{College of Computing}, \orgname{Georgia Institute of Technology}, \orgaddress{\city{Atlanta}, \state{Georgia}, \country{USA}}}
\affil[2]{\orgdiv{Department of Mathematics and Systems Engineering}, \orgname{Florida Institute of Technology}, \orgaddress{\city{Melbourne}, \state{Florida}, \country{USA}}}
\abstract{We study the dependence of parabolic solution operators on a finite signed measure describing the delay law. For a positive self-adjoint generator with compact inverse, a weighted dyadic spectral sum characterizes the norm of the semigroup-integrated memory perturbation. An exact realization identity transfers the matching lower estimate to two positive point delays in a fixed linear equation, using a common smooth, finite-spectral history. For the Dirichlet Laplacian on a nonempty bounded open set, the sharp worst-case modulus on a bounded ball of continuous $L^2$-valued histories is $d\sqrt{\log(e/d)}$, where $d$ is the bounded-Lipschitz distance between the delay measures. This is a rough-history endpoint result: logarithmic spatial regularity of order $\gamma>1/2$ restores resolution-uniform Lipschitz stability in the linear model. The critical order $\gamma=1/2$ retains a square-root double-logarithmic loss. More generally, a reciprocal-square summability criterion over occupied spectral bands gives the exact weighted threshold, including sparse spectra. Consequences include sharp finite-resolution Lipschitz constants, worst-case errors for prescribed midpoint quadrature of the memory measure, and semilinear upper estimates under Hilbert-space-valued local Lipschitz assumptions. The results distinguish sensitivity to the delay law from spatial approximation error and do not preclude uniform approximation of positive-time states on rough history balls.}
\maketitle
\raggedbottom

\section{Introduction}\label{sec:intro}
How does the geometry of the spatial spectrum determine sensitivity to a measure-valued delay? We study this question for parabolic equations of the form
\begin{equation}\label{eq:intro-pde}
 \partial_tu+Au+f(u)
 =\int_{-r}^0 B(u(t+\theta))\dd\mu(\theta)+h,
 \qquad u|_{[-r,0]}=\phi,
\end{equation}
where $A$ is positive and self-adjoint on a Hilbert space $H$, and $\mu$ is a finite signed Borel measure. The principal results concern the linear identity-feedback case $f=0$, $B=\mathrm{Id}$, $h=0$. The prototype is the Dirichlet Laplacian on $L^2(\Omega)$. We normalize the lowest eigenvalue to one. Histories are continuous with values in $H$; boundedness of a history class does not impose a common temporal modulus or spatial regularity bound.

For the Dirichlet Laplacian we prove the sharp worst-case modulus
\begin{equation}\label{eq:intro-modulus}
 \Psi(d)=d\sqrt{\log(e/d)},\qquad 0<d\leq1,
\end{equation}
where $d=d_{\BL}(\mu,\nu)$ is the bounded-Lipschitz distance. Sharpness means a uniform upper estimate on the specified history ball and a matching lower estimate for a family of histories and kernel pairs. It does not mean a two-sided estimate for every pair of kernels. The lower estimate is realized with two positive, unit-mass point delays, so neither nonlinear instability nor signed cancellation is responsible for the loss.

The regularity qualification is part of the main result. In the linear model, a uniform bound in $C([-r,0];D((1+\log A)^\gamma))$ restores Lipschitz dependence when $\gamma>1/2$; at $\gamma=1/2$ a square-root double-logarithmic loss remains. Every uniform positive spatial-power bound also suffices. Thus~\eqref{eq:intro-modulus} describes a rough-history endpoint, not a universal barrier on regular data classes. Sensitivity estimates on more regular history classes must be assessed using the corresponding weighted modulus. The precise threshold, including its dependence on the occupied spectral bands, is one of the main conclusions.

The semigroup formulation of partial functional differential equations is classical; see Travis and Webb~\cite{travis1974}, Wu~\cite{wu1996}, and B\'atkai and Piazzera~\cite{batkai2005}. Casas, Mateos and Tr\"oltzsch~\cite{casas2018} study measure control of a semilinear parabolic equation, including differentiability and control-space approximation. Kryspin and Mierczy\'nski~\cite{kryspin2024} treat regularization and weak-topology parameter dependence for linear parabolic delay systems. Shikhman~\cite{shikhman2026} develops the finite signed-measure framework for reaction--diffusion equations, with total-variation stability, weak-star convergence, and attractor upper semicontinuity. The present paper retains that measure formulation but addresses a different question: the sharp spectral modulus and the spatial regularity required for Lipschitz dependence. For completeness, we give self-contained solution constructions and all quantitative arguments used below.

\subsection{Spectral characterization and exact realization}
Our basic operator is
\[
 \mathcal K_a g=\int_0^\infty(I-e^{-aA})e^{-sA}g(s)\dd s.
\]
If $\mathcal J(A)$ is the set of occupied dyadic spectral bands and $w$ is a nondecreasing doubling weight, Theorem~\ref{thm:spectral} proves
\begin{equation}\label{eq:intro-profile}
 \norm{\mathcal K_a}_{L^\infty(D(w(A)))\to H}
 \asymp
 \left(\sum_{k\in\mathcal J(A)}
       \frac{\min\{a,2^{-k}\}^2}{w(2^k)^2}\right)^{1/2}.
\end{equation}
The constants are independent of dimension, eigenvalue multiplicity, and the locations of the occupied bands. Theorem~\ref{thm:lower} then realizes this profile inside the delay equation. A common history is chosen so that the difference at a fixed observation time satisfies the exact identity
\begin{equation}\label{eq:intro-realization}
 u(T)-v(T)=\int_0^{T/2}(I-e^{-aA})e^{-sA}h_a(s)\dd s.
\end{equation}
There are no endpoint remainders. The pulse $h_a$ is smooth, compactly supported in time, and finite-spectral. This identity makes the lower bound a statement about actual parabolic solutions, not only an auxiliary convolution norm.

Orthogonality across spectral bands explains the square root in~\eqref{eq:intro-modulus}. Integrating the usual operator-norm estimate
\begin{equation}\label{eq:intro-crude}
 \norm{(I-e^{-aA})e^{-sA}}_{\mathcal L(H)}
 \lesssim\min\{1,a/s\}e^{-s/2}
\end{equation}
would instead give $a\log(e/a)$. Standard analytic-semigroup bounds are treated in~\cite{pazy1983}; the square-function viewpoint belongs to the classical Littlewood--Paley theory for semigroups~\cite{stein1970}. The scalar modulus~\eqref{eq:intro-modulus}, like $d\log(e/d)$, satisfies the classical Osgood condition $\int_{0}^{1}\Psi(d)^{-1}\dd d=\infty$~\cite{osgood1898}. Neither the use of a logarithmic modulus nor the square-function principle is claimed as new. The quantitative content here is the occupied-band formula and its exact realization as a perturbation of the delay law.

The endpoint failure is also consistent with the established maximal-regularity obstruction for unbounded generators on Hilbert spaces. Baillon's theorem is discussed by Eberhardt and Greiner~\cite{eberhardt1992}; refinements and connections to admissibility appear in Jacob, Schwenninger and Wintermayr~\cite{jacob2022} and Preu\ss ler and Schwenninger~\cite{preussler2026}. Formula~\eqref{eq:intro-profile} quantifies the present perturbation family. It yields the necessary and sufficient condition
\begin{equation}\label{eq:intro-threshold}
 \sum_{k\in\mathcal J(A)}w(2^k)^{-2}<\infty
\end{equation}
for Lipschitz kernel dependence on the weighted history ball. Theorem~\ref{thm:log-hierarchy} identifies the critical logarithmic order for elliptic spectra; sparse spectra can have different thresholds and rates.

\subsection{Approximation consequences and their scope}
For a spectral truncation with largest eigenvalue $\lambda_N$, elliptic spectral density reduces the unweighted profile to
\begin{equation}\label{eq:intro-phase}
 a\sqrt{1+\log\!\bigl(\min\{\lambda_N,a^{-1}\}\bigr)}.
\end{equation}
The optimal Lipschitz constant therefore grows as $\sqrt{1+\log\lambda_N}$ on the rough history ball. Without spectral density the correct quantity is the square root of the number of occupied bands, not the spectral radius.

Numerical approximation of delay and Volterra equations has a substantial separate literature. Brunner~\cite{brunner2004} develops collocation for Volterra and related functional differential equations; Breda, Maset and Vermiglio~\cite{breda2015} study pseudospectral approximation of delay generators and solution operators for stability analysis. Pr\"uss~\cite{pruss1993} treats evolutionary Volterra equations with operator-valued kernels. We do not replace those analyses: our quadrature theorem concerns the propagated error when a finite delay measure is replaced by cell masses at prescribed midpoint nodes. Its worst-case rate is~\eqref{eq:intro-modulus} with $d$ replaced by the cell width. Adaptive recovery of atoms and higher-order approximation of a fixed smooth kernel are outside that assertion.

Classical spatial approximation estimates, such as those in Thom\'ee~\cite{thomee2006} and Auestad~\cite{auestad2026}, concern another error. We give a joint estimate that retains the initial-history projection defect, and we separately show that the commuting linear model admits uniform approximation of its positive-time states even on the rough history ball. Hence growing kernel sensitivity and convergence of spatial approximations can coexist. The semilinear extension is an upper-bound application under $H$-valued bounded-ball Lipschitz assumptions; it is not a sharp theorem for arbitrary nonlinear reactions.

Sections~\ref{sec:spectral}--\ref{sec:proof-spectral} establish the spectral norm formula. Section~\ref{sec:measure} transfers it to signed-measure perturbations. Section~\ref{sec:lower} constructs the linear solution operators and proves the exact lower-bound realization; Section~\ref{sec:regularity} gives the regularity threshold. Section~\ref{sec:semilinear} supplies the self-contained semilinear upper-bound extension. Section~\ref{sec:discrete} treats spectral resolution, prescribed memory quadrature, and spatial approximation, ending with a comparison between uniform approximation accuracy and Lipschitz stability in the delay law.

\section{Spectral setting and the principal norm formula}\label{sec:spectral}
Let $H$ be a separable real Hilbert space. We allow both finite and infinite dimension. Let $A$ be positive and self-adjoint, with an orthonormal eigenbasis $(e_j)$ and eigenvalues
\begin{equation}\label{eq:eigenvalues}
 1=\lambda_1\leq\lambda_2\leq\cdots.
\end{equation}
In infinite dimension, assume $A^{-1}$ is compact, so $\lambda_j\to\infty$. The normalization $\lambda_1=1$ fixes the time unit. For an unnormalized operator, the spectral ratios below are $\lambda_j/\lambda_1$, and time displacements are measured in units of $\lambda_1^{-1}$. Put $E(t)=e^{-tA}$. Thus $\norm{E(t)}\leq e^{-t}$.

For $k\in\Nzero$, define the spectral projection
\begin{equation}\label{eq:bands}
 P_k=\mathbf 1_{[2^k,2^{k+1})}(A),
 \qquad \calJ(A)=\{k\in\Nzero:P_k\ne0\}.
\end{equation}
These are band projections, not the Galerkin projections used later. In particular $0\in\calJ(A)$.

\begin{definition}[Admissible spatial weight]\label{def:weight}
A weight is a nondecreasing Borel function $w:[1,\infty)\to[1,\infty)$ satisfying $w(1)=1$ and
\begin{equation}\label{eq:doubling}
 w(2x)\leq D_w w(x),\qquad x\geq1,
\end{equation}
for a finite constant $D_w$. Set
\[
 H_w=D(w(A)),\qquad \norm{x}_{H_w}=\norm{w(A)x}_H,
 \qquad w_k=w(2^k).
\]
The norm is complete, and $w(A)^{-1}$ is an isometry from $H$ onto $H_w$.
\end{definition}
Examples are $w\equiv1$, $w(x)=x^\beta$ for $\beta>0$, and
\begin{equation}\label{eq:log-weight}
 w_\gamma(x)=(1+\log x)^\gamma,\qquad \gamma\geq0.
\end{equation}
Define, for $a\geq0$,
\begin{equation}\label{eq:rho}
 \rhoAw(a)=\left(\sum_{k\in\calJ(A)}
       \frac{\min\{a,2^{-k}\}^2}{w_k^2}\right)^{1/2},
 \qquad \rhoA=\rho_{A,1}.
\end{equation}
The sum converges for every $a$. For $0<a\leq1$,
\begin{equation}\label{eq:rho-properties}
 a\leq\rhoAw(a)\leq\rhoA(a)\leq C a\sqrt{\log(e/a)}.
\end{equation}
Moreover, $\rhoAw$ is increasing and $\rhoAw(a)/a$ is nonincreasing on $(0,\infty)$. In particular,
\begin{equation}\label{eq:rho-scale}
 \min\{1,c\}\rhoAw(a)\leq\rhoAw(ca)
 \leq\max\{1,c\}\rhoAw(a),\qquad c>0.
\end{equation}
The first lower bound in~\eqref{eq:rho-properties} uses the occupied band $k=0$. The upper bound follows by separating the sum at $\lfloor\log_2(a^{-1})\rfloor$. All other assertions follow term by term from~\eqref{eq:rho}.

The basic operator is
\begin{equation}\label{eq:K}
 \calK_a g=\int_0^\infty (I-E(a))E(s)g(s)\dd s,
 \qquad g\in L^\infty(0,\infty;H_w).
\end{equation}
This Bochner integral is well defined in $H$, since $\norm{(I-E(a))E(s)}\leq e^{-s}$. Its sharp norm will be denoted by $\kappa_{A,w}(a)$.

\begin{theorem}[Spectral-scale characterization]\label{thm:spectral}
For every admissible weight and every $a>0$,
\begin{equation}\label{eq:exact-dual}
 \kappa_{A,w}(a)
 =\sup_{\norm{v}_H=1}\int_0^\infty
   \norm{w(A)^{-1}(I-E(a))E(s)v}_H\dd s.
\end{equation}
There are constants $c_w,C_w>0$, depending only on $D_w$, such that
\begin{equation}\label{eq:spectral-norm}
 c_w\rhoAw(a)\leq\kappa_{A,w}(a)\leq C_w\rhoAw(a).
\end{equation}
The same operator norm is obtained by taking the supremum only over smooth, compactly supported functions of time whose values lie in a finite-dimensional spectral subspace. The constants are independent of the dimension, the eigenvalue multiplicities, and the locations of the nonempty bands.
\end{theorem}

The use of occupied bands is essential. Formula~\eqref{eq:spectral-norm} contains no spectral-density hypothesis. Such a hypothesis is needed only to replace the spectral sum by a logarithm of the largest eigenvalue.

\begin{definition}[Bounded gaps between occupied bands]\label{def:gaps}
An infinite-dimensional operator satisfies the band-gap condition if there exist integers $k_0\geq0$ and $q\geq1$ such that
\begin{equation}\label{eq:gaps}
 \calJ(A)\cap\{k,k+1,\ldots,k+q-1\}\ne\varnothing
 \quad\hbox{for every }k\geq k_0.
\end{equation}
\end{definition}
For $w=1$, this condition implies
\begin{equation}\label{eq:rho-dense}
 \rhoA(a)\asymp a\sqrt{\log(e/a)},\qquad 0<a\leq1.
\end{equation}
Indeed, a fixed fraction of the bands below $\log_2(a^{-1})$ is occupied; the finitely many low bands are absorbed into the constants. The general weighted consequences are proved in Section~\ref{sec:regularity}.

\begin{proposition}[Elliptic example]\label{prop:elliptic}
Let $\Omega\subset\R^n$ be a nonempty bounded open set. Let the Dirichlet Laplacian be the positive self-adjoint operator associated with the form $\int_\Omega\nabla u\cdot\nabla v$ on $H_0^1(\Omega)$. Divided by its first eigenvalue, it satisfies~\eqref{eq:gaps}. Its eigenvalues satisfy
\begin{equation}\label{eq:eigen-growth}
 c_\Omega j^{2/n}\leq\lambda_j\leq C_\Omega j^{2/n}.
\end{equation}
Thus~\eqref{eq:rho-dense} holds in every spatial dimension.
\end{proposition}
\begin{proof}
Choose cubes $Q_-\subset\Omega\subset Q_+$. Zero extension into $Q_+$ gives compactness of $H_0^1(\Omega)\hookrightarrow L^2(\Omega)$ and a positive Poincar\'e constant, so the form operator has compact inverse. No boundary regularity is needed. Choose the cubes with positive side lengths. Domain monotonicity from the min--max principle bounds the Dirichlet eigenvalues of $\Omega$ between those of these cubes. Counting the integer lattice points in the explicit cube spectra gives~\eqref{eq:eigen-growth}. These standard spectral principles are treated in~\cite{davies1995}. For completeness, the resulting band-gap argument does not require a Weyl remainder: for $R\geq1$, choose the first $j$ for which $\lambda_j\geq R$. If $j>1$, then $c_\Omega(j-1)^{2/n}\leq\lambda_{j-1}<R$, so~\eqref{eq:eigen-growth} gives $\lambda_j\leq C'_\Omega R$. Increasing $C'_\Omega$ handles $j=1$. Choose $q$ with $2^q>C'_\Omega$. The interval $[2^k,2^{k+q})$ then contains an eigenvalue, which proves~\eqref{eq:gaps}.
\end{proof}

\section{Proof of the spectral-scale characterization}\label{sec:proof-spectral}
Write
\[
 B_a(s)=(I-E(a))E(s),\qquad s\geq0.
\]
All the operators in this expression commute with $w(A)^{-1}$ and with the band projections.

\subsection{Duality and smooth test functions}
For $\norm g_{L^\infty(H_w)}\leq1$ and $\norm v_H=1$, self-adjointness gives
\[
 \abs{\ip{\calK_a g}{v}}
 \leq\int_0^\infty\norm{w(A)^{-1}B_a(s)v}_H\dd s.
\]
Conversely, put $z_v(s)=w(A)^{-1}B_a(s)v$. If $v\ne0$ and $a>0$, then $z_v(s)\ne0$ at every finite $s\geq0$: each spectral multiplier is strictly positive. The measurable function
\begin{equation}\label{eq:dual-input}
 g_v(s)=w(A)^{-1}\frac{z_v(s)}{\norm{z_v(s)}_H}
\end{equation}
has $H_w$-norm one and satisfies
\[
 \ip{\calK_a g_v}{v}=\int_0^\infty\norm{z_v(s)}_H\dd s.
\]
Taking the supremum over $v$ proves~\eqref{eq:exact-dual}.

The functional $J_a(v)=\int_0^\infty\norm{w(A)^{-1}B_a(s)v}\dd s$ is Lipschitz on $H$, since
\[
 |J_a(v)-J_a(\widetilde v)|
 \leq\norm{v-\widetilde v}\int_0^\infty e^{-s}\dd s
 =\norm{v-\widetilde v}.
\]
Finite spectral vectors are dense, so its supremum on the unit sphere can be approximated by such vectors. For finite spectral $v$, the function in~\eqref{eq:dual-input} is smooth at every finite time. Multiply it by smooth scalar cutoffs with compact support in $(0,\infty)$ and values in $[0,1]$, tending pointwise to one. Dominated convergence preserves $J_a(v)$ in the limit. This proves the assertion about smooth, compactly supported, finite-spectral inputs.

\subsection{The upper estimate}
Set
\[
 b_k=\frac{\min\{a,2^{-k}\}}{w_k}.
\]
For $\lambda\in[2^k,2^{k+1})$,
\[
 \frac{(1-e^{-a\lambda})e^{-s\lambda}}{w(\lambda)}
 \leq\frac{\min\{1,a2^{k+1}\}e^{-s2^k}}{w_k}.
\]
Integration therefore gives
\begin{equation}\label{eq:band-upper}
 \int_0^\infty\norm{w(A)^{-1}B_a(s)P_kv}\dd s
 \leq 2b_k\norm{P_kv}.
\end{equation}
The triangle inequality and Cauchy--Schwarz across the orthogonal bands imply
\[
 J_a(v)\leq 2\sum_{k\in\calJ(A)}b_k\norm{P_kv}
 \leq2\rhoAw(a)\norm v.
\]
The argument first applies to finite spectral sums and then to all $v$ by continuity. This proves the upper half of~\eqref{eq:spectral-norm}.

\subsection{The matching lower estimate}
Take a finite subset of $\calJ(A)$. Retain either its even or its odd indices, choosing a set $F$ that carries at least half the sum of $b_k^2$. Choose one unit eigenvector $f_k$ from each retained band, with eigenvalue $\eta_k\in[2^k,2^{k+1})$, and set
\begin{equation}\label{eq:lower-v}
 Z_F=\left(\sum_{k\in F}b_k^2\right)^{1/2},
 \qquad v_F=\frac{1}{Z_F}\sum_{k\in F}b_k f_k.
\end{equation}
The intervals $J_k=[\eta_k^{-1},2\eta_k^{-1}]$ are pairwise disjoint up to endpoints. Indeed, successive retained bands differ by at least two, so their selected eigenvalues differ by a factor greater than two. On $J_k$, retain just the $f_k$ coordinate in the norm. It follows that
\begin{align}
 \int_{J_k}\norm{w(A)^{-1}B_a(s)v_F}\dd s
 &\geq\frac{b_k}{Z_F}
     \frac{(1-e^{-a\eta_k})(e^{-1}-e^{-2})}{\eta_k w(\eta_k)}
       \notag\\
 &\geq\frac{c(D_w)b_k^2}{Z_F}.\label{eq:band-lower}
\end{align}
For the last inequality use $1-e^{-x}\geq(1-e^{-1})\min\{x,1\}$, $\eta_k<2^{k+1}$, and $w(\eta_k)\leq D_w w_k$. Summing over the disjoint intervals yields $J_a(v_F)\geq c(D_w)Z_F$. The parity selection loses at most a factor $\sqrt2$. Exhausting $\calJ(A)$ by finite sets proves the lower bound and completes the proof of Theorem~\ref{thm:spectral}.

\begin{corollary}[The weighted endpoint]\label{cor:endpoint}
For every admissible weight,
\begin{equation}\label{eq:endpoint-criterion}
 \sup_{0<a\leq1}\frac{\kappa_{A,w}(a)}{a}
 \asymp\left(\sum_{k\in\calJ(A)}w_k^{-2}\right)^{1/2},
\end{equation}
where the value $+\infty$ is allowed. In particular, for $w=1$ a finite-dimensional operator has an endpoint constant comparable to the square root of its number of occupied bands.
\end{corollary}
\begin{proof}
After division by $a^2$, every summand in~\eqref{eq:rho} increases to $w_k^{-2}$ as $a\downarrow0$. Monotone convergence and Theorem~\ref{thm:spectral} prove~\eqref{eq:endpoint-criterion}.
\end{proof}

\begin{remark}[A square-function explanation of the logarithm]\label{rem:square}
The unweighted continuum upper bound also has a direct Hilbert-space proof, in the square-function tradition of semigroup Littlewood--Paley theory~\cite{stein1970}. The spectral theorem gives
\begin{align*}
 \int_0^\infty s\norm{B_a(s)v}^2\dd s
 &=\frac14\sum_j\frac{(1-e^{-a\lambda_j})^2}{\lambda_j^2}
                   |\ip v{e_j}|^2
 \leq\frac{a^2}{4}\norm v^2.
\end{align*}
For $0<a<1$, weighted Cauchy--Schwarz on $(a,1)$ therefore bounds $\int_a^1\norm{B_a(s)v}\dd s$ by $(a/2)\sqrt{\log(1/a)}\norm v$. The intervals $(0,a)$ and $(1,\infty)$ contribute $O(a)\norm v$, using contraction on the first interval and $1-e^{-a\lambda}\leq a\lambda$ with exponential decay on the second. The dyadic formula refines this argument by retaining the actual occupied scales and the input weight. The qualitative unweighted divergence is consistent with the maximal-regularity obstruction discussed in~\cite{eberhardt1992,jacob2022,preussler2026}; the formula specifies its rate for the present family of perturbations.
\end{remark}

\section{Delay regularization and signed-measure perturbations}\label{sec:measure}
Let $I=[-r,0]$, $r>0$, and write $\calM(I)$ for the finite signed Borel measures. The total variation is $\norm\mu_{\TV}=|\mu|(I)$. We use
\begin{equation}\label{eq:BL}
 d_{\BL}(\mu,\nu)=
 \sup_{\norm q_\infty+\Lip(q)\leq1}
      \abs{\int_I q\dd(\mu-\nu)},
\end{equation}
where the tests are real-valued. This convention, using the sum rather than the maximum of the two test norms, determines the explicit atom distance in Lemma~\ref{lem:dirac}.

On each total-variation-bounded set, $d_{\BL}$ metrizes weak-star convergence. To verify this directly, the scalar test ball in~\eqref{eq:BL} is compact in $C(I)$ by Arzel\`a--Ascoli. A finite-net argument and the common variation bound turn pointwise weak-star convergence into uniform convergence over that ball. Conversely, piecewise affine functions are uniformly dense in $C(I)$, and the common variation bound controls the approximation error. No positivity assumption is used.

For $g\in C([-r,T];H_w)$, define
\begin{equation}\label{eq:G}
 G_{t,g}(\theta)=\int_0^t E(t-s)g(s+\theta)\dd s,
 \qquad 0\leq t\leq T,\quad\theta\in I.
\end{equation}

\begin{lemma}[Regularization in the delay coordinate]\label{lem:G}
There is $C=C(D_w)$, independent of $T$ and of spectral truncation, such that
\begin{equation}\label{eq:G-sup}
 \norm{G_{t,g}}_{C(I;H)}\leq\norm g_{C([-r,T];H_w)}
\end{equation}
and
\begin{equation}\label{eq:G-mod}
 \norm{G_{t,g}(\eta)-G_{t,g}(\theta)}_H
 \leq C\norm g_{C([-r,T];H_w)}\rhoAw(|\eta-\theta|).
\end{equation}
\end{lemma}
\begin{proof}
The first estimate follows from $\norm{E(s)}\leq e^{-s}$ and $\norm{x}_H\leq\norm{x}_{H_w}$. Assume $\eta=\theta+a$ with $a>0$. Changing variables gives
\[
 G_{t,g}(\theta)=\int_\theta^{t+\theta}E(t+\theta-q)g(q)\dd q.
\]
If $a<t$, the two nonoverlapping endpoint intervals contribute at most $2a\norm g_\infty$. On the overlap $q\in[\theta+a,t+\theta]$, their difference is
\[
 -\int_0^{t-a}(I-E(a))E(s)g(t+\theta-s)\dd s.
\]
Extending this input by zero outside $[0,t-a]$, Theorem~\ref{thm:spectral} bounds its norm by $C\rhoAw(a)\norm g_\infty$. When $0<a\leq1$, the endpoint terms are absorbed using $\rhoAw(a)\geq a$. If $t\leq a\leq1$, the separate integral lengths instead give $2t\norm g_\infty\leq2a\norm g_\infty$. For $a>1$, use~\eqref{eq:G-sup} and $\rhoAw(a)\geq1$. These cases prove~\eqref{eq:G-mod}.
\end{proof}

\begin{theorem}[Integrated perturbation of signed measures]\label{thm:measure}
For $\norm\mu_{\TV},\norm\nu_{\TV}\leq M$,
\begin{equation}\label{eq:measure-bound}
 \sup_{0\leq t\leq T}
 \norm{\int_I G_{t,g}(\theta)\dd(\mu-\nu)(\theta)}_H
 \leq C\norm g_{C([-r,T];H_w)}\rhoAw(d_{\BL}(\mu,\nu)),
\end{equation}
where $C$ depends only on $M,r,D_w$. In particular, it is independent of the time horizon and of spatial resolution.
\end{theorem}
\begin{proof}
Write $d=d_{\BL}(\mu,\nu)$ and $d_0=\min\{1,r\}$. The assertion is immediate when $d=0$. Suppose $0<d\leq d_0$, and partition $I$ into $n=\lceil r/d\rceil$ equal intervals of length $\ell=r/n\in[d/2,d]$. Let $G_\ell$ be the piecewise affine interpolant of $G=G_{t,g}$. Lemma~\ref{lem:G} implies
\begin{align*}
 \norm{G-G_\ell}_\infty&\leq C\norm g_\infty\rhoAw(d),\\
 \norm{G_\ell}_\infty&\leq\norm g_\infty,\\
 \Lip_H(G_\ell)&\leq C\norm g_\infty\frac{\rhoAw(d)}{d}.
\end{align*}
Testing against unit vectors in $H$ and using~\eqref{eq:BL} yields
\[
 \norm{\int_I G_\ell\dd(\mu-\nu)}_H
 \leq d\bigl(\norm{G_\ell}_\infty+\Lip_H(G_\ell)\bigr).
\]
The interpolation remainder contributes at most $2M\norm{G-G_\ell}_\infty$. Since $d\leq\rhoAw(d)$ for $d\leq1$, these estimates prove~\eqref{eq:measure-bound} for $d\leq d_0$. For $d>d_0$, the direct bound is $2M\norm g_\infty$, whereas $\rhoAw(d)\geq\rhoAw(d_0)\geq d_0$. Increasing the constant proves the remaining case.
\end{proof}

\begin{remark}[The instantaneous map is not uniformly regularized]\label{rem:instantaneous}
Lemma~\ref{lem:G} concerns the time-integrated response, not the raw memory functional on an arbitrary bounded subset of $C(I;H)$. In particular, one cannot apply bounded-Lipschitz duality directly to a merely bounded continuous history and assume a uniform Lipschitz test norm. The semigroup integration is precisely what supplies the modulus used in Theorem~\ref{thm:measure}.
\end{remark}

\section{Exact realization in the linear delay equation}\label{sec:lower}
Set $X=C(I;H)$, $X_w=C(I;H_w)$, and $Y_T=C([-r,T];H)$. The linear reference equation is
\begin{equation}\label{eq:linear-model}
 \partial_tu+Au=\int_Iu(t+\theta)\dd\mu(\theta),\qquad u|_I=\phi.
\end{equation}
Write $U_\mu(t)\phi=u_\mu^\phi(t)$ for its current-state map. We first construct the solution and its weighted stability directly, so the sharp linear results do not depend on the later semilinear extension.

\begin{proposition}[Linear solution and weighted upper bound]\label{prop:linear}
For every $\mu\in\mathcal M(I)$ and $\phi\in X_w$, equation~\eqref{eq:linear-model} has a unique global mild solution with values in $H_w$. If $\norm\mu_{\TV}\leq M$, then
\begin{equation}\label{eq:linear-weighted}
 \norm{u_\mu^\phi}_{C([-r,T];H_w)}\leq e^{MT}\norm\phi_{X_w}.
\end{equation}
For $\norm\mu_{\TV},\norm\nu_{\TV}\leq M$ and $\norm\phi_{X_w},\norm\psi_{X_w}\leq R_w$,
\begin{equation}\label{eq:linear-stability}
 \norm{u_\mu^\phi-u_\nu^\psi}_{Y_T}
 \leq C\bigl(\norm{\phi-\psi}_{X}+R_w\rhoAw(d_{\BL}(\mu,\nu))\bigr),
\end{equation}
where $C$ depends on $M,T,r,D_w$ and is independent of spectral truncation.
\end{proposition}
\begin{proof}
In $C([-r,t_0];H_w)$, with history fixed, the map
\[
 u(t)\longmapsto E(t)\phi(0)+\int_0^t E(t-s)
             \int_Iu(s+\theta)\dd\mu(\theta)\dd s
\]
has Lipschitz constant at most $Mt_0$, because $E(t)$ is a contraction on $H_w$. It is therefore a contraction for $Mt_0<1$; the case $M=0$ is immediate. This construction includes an atom at zero. Taking the running supremum of the weighted norm in the mild formula gives
\[
 Z_w(t)\leq\norm\phi_{X_w}+M\int_0^t Z_w(s)\dd s.
\]
Gronwall proves~\eqref{eq:linear-weighted} and permits continuation on every finite interval.

For $q=u_\mu^\phi-u_\nu^\psi$, subtract the two mild formulas and split the delay difference into its trajectory part and its measure part evaluated along $u_\nu^\psi$. By Fubini, the integrated measure part is the expression in Theorem~\ref{thm:measure} with $g=u_\nu^\psi$. Its supremum norm is at most $CR_we^{MT}\rhoAw(d_{\BL}(\mu,\nu))$. The trajectory part is bounded by $M\int_0^t\sup_{-r\leq s\leq\sigma}\norm{q(s)}\dd\sigma$. A second application of Gronwall proves~\eqref{eq:linear-stability}. All estimates use contractions and Theorem~\ref{thm:measure}, whose constants are independent of the dimension.
\end{proof}

\subsection{The exact history construction}
The lower construction is performed before newly generated states enter the delay. Fix
\begin{equation}\label{eq:time-choice}
 0<T<\tau<r,\qquad
 0<a\leq a_0:=\min\{1,T/4,(r-\tau)/2\}.
\end{equation}
For a common history $\phi$, consider
\begin{equation}\label{eq:lower-model}
 \partial_t u+Au=u(t-\tau),\qquad
 \partial_t v+Av=v(t-\tau-a),\qquad
 u|_I=v|_I=\phi.
\end{equation}
Both kernels are positive probability measures. Define
\begin{equation}\label{eq:D-response}
 \mathfrak D_{A,w}(a;T,\tau)
 =\sup_{\norm\phi_{X_w}\leq1}\norm{u(T)-v(T)}_H.
\end{equation}

\begin{theorem}[Exact realization of the spectral modulus]\label{thm:lower}
Under~\eqref{eq:time-choice},
\begin{equation}\label{eq:lower-response}
 c\rhoAw(a)\leq\mathfrak D_{A,w}(a;T,\tau)
             \leq C\rhoAw(a).
\end{equation}
The constants depend on $T,r,D_w$, but are independent of the dimension and the spectrum. In the lower bound the histories may be chosen smooth in time and supported on finitely many eigenvectors, with $\phi(0)=0$.
\end{theorem}
\begin{proof}
For $0\leq t\leq T$, both delayed arguments are negative. Therefore
\begin{equation}\label{eq:first-step}
 u(T)-v(T)=\int_0^T E(T-s)
             [\phi(s-\tau)-\phi(s-\tau-a)]\dd s.
\end{equation}
The upper bound follows from Lemma~\ref{lem:G}, applied to any bounded continuous extension of $\phi$ beyond time zero.

For the lower bound, choose $k_T\geq0$ so that $2^{k_T}\geq8/T$, and put
\[
 \rho_{\mathrm{hi}}(a)^2
 =\sum_{\substack{k\in\calJ(A)\\ k\geq k_T}}
       \frac{\min\{a,2^{-k}\}^2}{w_k^2}.
\]
Select a finite set of these high bands, retain a parity class, and construct $v_F$ as in~\eqref{eq:lower-v}. The intervals $J_k=[\eta_k^{-1},2\eta_k^{-1}]$ now lie in $(0,T/4]$. Choose $\chi\in C_c^\infty((0,T/2))$ with $0\leq\chi\leq1$ and $\chi=1$ on all the selected intervals. Define, with zero extension to the real line,
\begin{equation}\label{eq:h-lower}
 h_a(s)=\chi(s)w(A)^{-1}
 \frac{w(A)^{-1}B_a(s)v_F}{\norm{w(A)^{-1}B_a(s)v_F}}.
\end{equation}
This function is smooth, takes values in a finite spectral space, and has $H_w$-norm at most one. Set
\begin{equation}\label{eq:phi-lower}
 \phi_a(\theta)=-h_a(T-\tau-a-\theta),\qquad \theta\in I.
\end{equation}
Its support is contained in $(T/2-\tau-a,T-\tau-a)\subset(-r,0)$, so it is an admissible smooth history with $\phi_a(0)=0$.

Writing $q=T-s$ in~\eqref{eq:first-step}, the difference becomes
\[
 \int_0^T E(q)[h_a(q)-h_a(q-a)]\dd q.
\]
Because $\supp h_a\subset(0,T/2)$ and $a\leq T/4$, both terms contain the entire support after the change of variables. Consequently there is the exact identity
\begin{equation}\label{eq:exact-realization}
 u(T)-v(T)=\int_0^{T/2}(I-E(a))E(s)h_a(s)\dd s.
\end{equation}
There are no endpoint remainder terms. Taking the inner product with $v_F$ and repeating~\eqref{eq:band-lower} gives a lower bound by the spectral sum over the selected high bands. Approximation by finite sets and the parity argument show that $\mathfrak D_{A,w}\geq c(D_w)\rho_{\mathrm{hi}}(a)$.

It remains to control the finitely many low bands. Use the separate history $\phi_0(\theta)=(\theta/r)e_1$, whose $X_w$-norm is one. Then~\eqref{eq:first-step} gives
\begin{equation}\label{eq:low-mode}
 u(T)-v(T)=\frac{a(1-e^{-T})}{r}e_1.
\end{equation}
Since $w_k\geq1$, the omitted part of $\rhoAw(a)^2$ is at most $k_Ta^2$. The two lower bounds, one using~\eqref{eq:phi-lower} and one using $\phi_0$, imply the full lower estimate in~\eqref{eq:lower-response}. The histories need not be the same for these two estimates because~\eqref{eq:D-response} is a supremum. All constants are uniform under spectral truncation.
\end{proof}

\begin{lemma}[Distance between two atoms]\label{lem:dirac}
For distinct $\theta,\eta\in I$, with $a=|\theta-\eta|$,
\begin{equation}\label{eq:dirac-distance}
 d_{\BL}(\delta_\theta,\delta_\eta)=\frac{2a}{2+a}.
\end{equation}
\end{lemma}
\begin{proof}
If $b=|q(\theta)-q(\eta)|$, then $\norm q_\infty\geq b/2$ and $\Lip(q)\geq b/a$. Thus the constraint in~\eqref{eq:BL} gives $b\leq2a/(2+a)$. Equality is attained by an affine function between the two points with values $\pm a/(2+a)$, extended constantly outside that interval.
\end{proof}

Let $\calP(I)$ be the probability measures on $I$. For the linear equation with a general kernel, define the worst-case modulus
\begin{equation}\label{eq:Omega}
 \Omega_{A,w,T}(d)=
 \sup_{\substack{\mu,\nu\in\calP(I),\ d_{\BL}(\mu,\nu)\leq d\\
                  \norm\phi_{X_w}\leq1}}
       \norm{U_\mu(T)\phi-U_\nu(T)\phi}_H.
\end{equation}

\begin{corollary}[Sharp solution-map modulus]\label{cor:omega}
Fix $0<T<r/2$. For all sufficiently small $d>0$,
\begin{equation}\label{eq:omega-rho}
 \Omega_{A,w,T}(d)\asymp\rhoAw(d),
\end{equation}
with constants independent of spectral truncation. In particular, for the Dirichlet Laplacian and $w=1$,
\begin{equation}\label{eq:omega-sharp}
 \Omega_{A,1,T}(d)\asymp d\sqrt{\log(e/d)}.
\end{equation}
No $o(d\sqrt{\log(e/d)})$ modulus holds uniformly on this history ball.
\end{corollary}
\begin{proof}
The upper estimate follows from Proposition~\ref{prop:linear}. For the lower estimate choose $\tau=r/2$ and $a=d$ in Theorem~\ref{thm:lower}. By Lemma~\ref{lem:dirac}, the distance between the two kernels is at most $d$. Proposition~\ref{prop:elliptic} gives~\eqref{eq:omega-sharp}.
\end{proof}

\begin{remark}[Quantifiers in the sharpness statement]\label{rem:quantifiers}
The lower-bound history depends on the displacement and on the resolved spectral bands. Every individual witness is smooth and uses finitely many modes, but the witnesses have no common bound on temporal derivatives or on positive powers of $A$. The result therefore does not assert the same lower rate for one fixed history, for a uniformly regular history class, or for every nonlinear feedback. It asserts a worst-case obstruction for the rough history ball, already within the positive linear subclass of~\eqref{eq:intro-pde}.
\end{remark}

\section{The sharp spatial-regularity threshold}\label{sec:regularity}
The spectral formula yields a necessary and sufficient weighted condition, not only a sufficient positive-power assumption. For the linear delay equation define
\begin{equation}\label{eq:L-infinity}
 L_{\infty,w}(T)=
 \sup_{\substack{\mu,\nu\in\calP(I),\ \mu\ne\nu\\
                 \norm\phi_{X_w}\leq1}}
 \frac{\norm{U_\mu(T)\phi-U_\nu(T)\phi}_H}
      {d_{\BL}(\mu,\nu)},\qquad 0<T<r/2.
\end{equation}

\begin{theorem}[Weighted Lipschitz criterion]\label{thm:weighted-L}
For every admissible weight,
\begin{equation}\label{eq:weighted-L-criterion}
 L_{\infty,w}(T)\asymp
 \left(\sum_{k\in\calJ(A)}w_k^{-2}\right)^{1/2},
\end{equation}
where either side may be infinite. Equivalently, the continuum map is Lipschitz on the unit ball of $X_w$ precisely when the sum is finite; the same condition is equivalent to a common bound for the corresponding Lipschitz constants on all spectral truncations.
\end{theorem}
\begin{proof}
If the sum is finite, the upper bound follows from the weighted linear estimate and $\rhoAw(d)\leq d(\sum w_k^{-2})^{1/2}$ for $d\leq1$; for distances $d>1$, use~\eqref{eq:linear-weighted} and the fact that the sum is at least one. The lower bound follows from Theorem~\ref{thm:lower} and monotone convergence as $a\downarrow0$. For spectral truncations the same proof restricts the sum to the retained bands. The linear equation preserves each finite spectral space, and monotone convergence of these sums proves the final equivalence.
\end{proof}

\begin{theorem}[Logarithmic regularity hierarchy]\label{thm:log-hierarchy}
Assume~\eqref{eq:gaps}, and use the weight $w_\gamma$ in~\eqref{eq:log-weight}. For $0<a\leq1$,
\begin{equation}\label{eq:log-hierarchy}
 \rho_{A,w_\gamma}(a)\asymp
 \begin{cases}
 a[\log(e/a)]^{1/2-\gamma},&0\leq\gamma<1/2,\\[2pt]
 a\sqrt{\log(e+\log(1/a))},&\gamma=1/2,\\[2pt]
 a,&\gamma>1/2.
 \end{cases}
\end{equation}
These are sharp worst-case kernel-stability moduli for the linear delay equation on the unit ball of $C(I;D((1+\log A)^\gamma))$. In particular, resolution-uniform Lipschitz stability holds exactly when $\gamma>1/2$.
\end{theorem}
\begin{proof}
Let $m=\lfloor\log_2(a^{-1})\rfloor$. Formula~\eqref{eq:rho} gives
\begin{equation}\label{eq:weighted-sum}
 \rho_{A,w_\gamma}(a)^2
 =a^2\sum_{\substack{k\in\calJ(A)\\k\leq m}}
            (1+k\log2)^{-2\gamma}
 +\sum_{\substack{k\in\calJ(A)\\k>m}}
            4^{-k}(1+k\log2)^{-2\gamma}.
\end{equation}
The second sum is bounded by $Ca^2(1+m)^{-2\gamma}$. By the bounded-gap condition, the first sum is comparable to
\[
 a^2\sum_{k=0}^m(1+k)^{-2\gamma}.
\]
For the lower comparison, group indices into blocks of length $q$, select one occupied index in each block, and use comparability of the weights inside a fixed-length block; the term $k=0$ absorbs the finitely many initial blocks. The partial sums grow as $(1+m)^{1-2\gamma}$ when $\gamma<1/2$, as $\log(e+m)$ when $\gamma=1/2$, and remain between positive constants when $\gamma>1/2$. This proves~\eqref{eq:log-hierarchy}. Sharpness for the solution map follows from Corollary~\ref{cor:omega}, and the Lipschitz criterion also follows directly from Theorem~\ref{thm:weighted-L}.
\end{proof}

\begin{corollary}[Positive spatial powers]\label{cor:power}
For every $\beta>0$, histories bounded in $C(I;D(A^\beta))$ give Lipschitz kernel stability for the linear equation, uniformly over spectral truncations. If~\eqref{eq:gaps} holds, the optimal constant satisfies
\begin{equation}\label{eq:power-constant}
 L_{\infty,x^\beta}(T)\asymp\beta^{-1/2},\qquad 0<\beta\leq1,
\end{equation}
with comparison constants independent of $\beta$ in this range.
\end{corollary}
\begin{proof}
Here $w_k=2^{\beta k}$ and $D_w=2^\beta\leq2$ for $\beta\leq1$. The upper geometric sum is comparable to $\beta^{-1}$. Under~\eqref{eq:gaps}, one occupied index per block of a fixed length gives the matching lower sum, uniformly for $0<\beta\leq1$. Apply Theorem~\ref{thm:weighted-L}.
\end{proof}

\begin{remark}[Temporal regularity is a different sufficient condition]\label{rem:time-regular}
If a reference feedback path $g$ is uniformly Lipschitz from $[-r,T]$ into $H$, then the instantaneous residual satisfies
\[
 \norm{\int_I g(t+\theta)\dd(\mu-\nu)(\theta)}_H
 \leq [\norm g_\infty+\Lip_H(g)]d_{\BL}(\mu,\nu).
\]
The difference argument in Section~\ref{sec:semilinear} then gives Lipschitz stability directly. This condition concerns the entire reference feedback path, including positive times. It is not silently inferred from temporal Lipschitz continuity of the prescribed negative-time history.
\end{remark}

\section{Semilinear upper bounds}\label{sec:semilinear}
Set $V=D(A^{1/2})$ with norm $\norm{x}_V=\norm{A^{1/2}x}_H$, and let $V'$ be its dual with pivot $H$. Use the spaces $X$ and $Y_T$ defined in Section~\ref{sec:lower}.
For a trajectory $u$, its history segment is $u_t(\theta)=u(t+\theta)$. Throughout this section the forcing $h\in V'$ is independent of time.

\begin{assumption}\label{ass:reaction}
The maps $f,B:H\to H$ are Lipschitz on bounded balls. There are constants $a_f\in\R$ and $b_0,b_1,c_f\geq0$ such that
\begin{equation}\label{eq:coercive}
 \ip{f(x)}{x}\geq a_f\norm{x}^2-c_f,
 \qquad \norm{B(x)}\leq b_0+b_1\norm{x},\qquad x\in H.
\end{equation}
\end{assumption}
These are hypotheses on operators on $H$. In particular, a pointwise cubic reaction on $L^2(\Omega)$ is not included merely by virtue of being a polynomial. Only upper estimates are asserted for the general semilinear class. The matching lower bounds concern the linear subclass already treated in Section~\ref{sec:lower}.

Define
\[
 D_\mu u(t)=\int_I B(u(t+\theta))\dd\mu(\theta).
\]
A weak solution of~\eqref{eq:intro-pde} belongs to $Y_T\cap L^2(0,T;V)$, has derivative in $L^2(0,T;V')$, equals $\phi$ on $I$, and satisfies the equation in $V'$ almost everywhere. For $z=A^{-1}h\in V$, its equivalent shifted mild formula is
\begin{equation}\label{eq:shifted-mild}
 u(t)=z+E(t)(\phi(0)-z)
       +\int_0^t E(t-s)\bigl[-f(u(s))+D_\mu u(s)\bigr]\dd s.
\end{equation}
The use of the stationary shift keeps the common rough forcing out of all kernel-difference estimates.

\begin{proposition}[Finite-time solution framework]\label{prop:wellposed}
Under Assumption~\ref{ass:reaction}, every $\phi\in X$ and $\mu\in\calM(I)$ determine a unique global weak solution. For every $R,M,T<\infty$ there is $R_T<\infty$ such that
\begin{equation}\label{eq:uniform-state}
 \norm{u_\mu^\phi}_{Y_T}\leq R_T
 \quad\hbox{if }\norm\phi_X\leq R,
                   \quad\norm\mu_{\TV}\leq M.
\end{equation}
The same bound, with a common constant, holds for the spectral Galerkin equations with projected histories, forcing, reaction, and feedback.
\end{proposition}
\begin{proof}
For fixed history, the right side of~\eqref{eq:shifted-mild} maps a sufficiently large closed ball of continuous trajectories into itself on a short interval. On a ball of radius $R_*$ its trajectory-dependent part has Lipschitz constant at most
\[
 t_0\bigl[L_f(R_*)+ML_B(R_*)\bigr],
\]
since
\[
 \norm{D_\mu u(t)-D_\mu v(t)}
 \leq ML_B(R_*)\sup_{s\in[t-r,t]}\norm{u(s)-v(s)}.
\]
Strong continuity of $E(t)$ ensures that the nonintegral part remains in the chosen ball when $t_0$ is small. The contraction principle therefore gives a unique local mild solution, also when $\mu$ has an atom at zero.

Here is a direct justification of the weak regularity used in the energy argument. On the local interval put $y=u-z$ and $F=-f(u)+D_\mu u\in C([0,t_0];H)$. For the projection $\Pi_m$ onto the first $m$ eigenvectors, $y_m=\Pi_my$ solves the finite linear system $y_m'+Ay_m=\Pi_mF$. Testing by $y_m$ and using $\lambda_1=1$ yields
\[
 \frac{\mathrm d}{\mathrm dt}\norm{y_m}^2+\norm{y_m}_V^2
 \leq\norm F^2,\qquad
 \norm{y_m'}_{V'}\leq\norm{y_m}_V+\norm F.
\]
Apply the same estimates to $y_m-y_n$ with data $(\Pi_m-\Pi_n)y(0)$ and forcing $(\Pi_m-\Pi_n)F$. The projected data converge in $H$ and the projected forcing converges in $L^2(0,t_0;H)$. Thus $y_m$ converges in $C([0,t_0];H)\cap L^2(0,t_0;V)$ and $y_m'$ in $L^2(0,t_0;V')$. The limit is the mild solution, so $u=y+z$ has the asserted weak regularity. Its energy identity follows by passage to the limit. Conversely, projection of a weak solution gives the scalar variation-of-constants formulas and hence~\eqref{eq:shifted-mild}.

Test the equation by $u$, and set $Y(t)=\sup_{-r\leq s\leq t}\norm{u(s)}^2$. The bound $\norm{D_\mu u(t)}\leq M(b_0+b_1\sqrt{Y(t)})$, Young's inequality, and~\eqref{eq:coercive} imply
\begin{equation}\label{eq:energy-bound}
 \frac{\mathrm d}{\mathrm dt}\norm{u(t)}^2+\norm{u(t)}_V^2
 \leq C(1+Y(t)).
\end{equation}
After integration and taking the running supremum, Gronwall gives~\eqref{eq:uniform-state} and the $L^2(0,T;V)$ bound. On the resulting $H$-ball, $f(u)$ and $D_\mu u$ are bounded in $H$, so the equation also bounds the derivative in $L^2(0,T;V')$. If a maximal existence interval had a finite endpoint,~\eqref{eq:shifted-mild} with bounded forcing would give a continuous limit there; the local construction could then be restarted. This proves global existence. Subtracting the mild formulas and applying Gronwall proves uniqueness. Finally, spectral projections are contractions on $H,V,V'$ and preserve the coercivity inequality when tested against a projected vector. The same proof gives all bounds uniformly for the Galerkin equations.
\end{proof}

Use $U_\mu(t)\phi=u_\mu^\phi(t)$ also for the semilinear current-state map and write $S_\mu(t)\phi=(u_\mu^\phi)_t$ for the history map. The norm of a difference in $Y_T$ equals the supremum of the corresponding history differences over $0\leq t\leq T$.

\begin{theorem}[Spectral and weak-star stability]\label{thm:semilinear}
Fix $R,M,T<\infty$ and assume~\ref{ass:reaction}. There is a constant $C=C(R,M,T,f,B,h,r)$, independent of spectral truncation, such that
\begin{equation}\label{eq:semilinear-bound}
 \norm{u_\mu^\phi-u_\nu^\psi}_{Y_T}
 \leq C\bigl(\norm{\phi-\psi}_X+\rhoA(d_{\BL}(\mu,\nu))\bigr)
\end{equation}
for $\norm\phi_X,\norm\psi_X\leq R$ and $\norm\mu_{\TV},\norm\nu_{\TV}\leq M$. In particular, for $0<d=d_{\BL}(\mu,\nu)\leq1$,
\begin{equation}\label{eq:semilinear-psi}
 \sup_{0\leq t\leq T}\norm{S_\mu(t)\phi-S_\nu(t)\psi}_X
 \leq C\bigl(\norm{\phi-\psi}_X+d\sqrt{\log(e/d)}\bigr).
\end{equation}
For the $N$-mode system, $\rhoA$ in~\eqref{eq:semilinear-bound} may be replaced by $\rhoN$, with the same type of constant.
\end{theorem}
\begin{proof}
Let $u=u_\mu^\phi$, $v=u_\nu^\psi$, and $q=u-v$. The common forcing cancels. Subtraction of the mild formulas gives
\begin{align}
 q(t)={}&E(t)(\phi(0)-\psi(0))
 -\int_0^t E(t-s)[f(u(s))-f(v(s))]\dd s\notag\\
 &+\int_0^t E(t-s)[D_\mu u(s)-D_\mu v(s)]\dd s
 +R_{\mu,\nu}(t),\label{eq:diff-mild}\\
 R_{\mu,\nu}(t)={}&\int_0^t E(t-s)
       \int_I B(v(s+\theta))\dd(\mu-\nu)(\theta)\dd s.
       \label{eq:residual}
\end{align}
By Fubini,~\eqref{eq:residual} is the expression in Theorem~\ref{thm:measure}, with $g=B(v)$ and $w=1$. Proposition~\ref{prop:wellposed} gives $\norm g_\infty\leq b_0+b_1R_T$. Hence
\[
 \sup_{0\leq t\leq T}\norm{R_{\mu,\nu}(t)}
 \leq C\rhoA(d_{\BL}(\mu,\nu)).
\]
Put $Z(t)=\sup_{-r\leq s\leq t}\norm{q(s)}$. The local Lipschitz constants on the common ball imply
\[
 Z(t)\leq\norm{\phi-\psi}_X+C\rhoA(d_{\BL}(\mu,\nu))
    +[L_f(R_T)+ML_B(R_T)]\int_0^t Z(s)\dd s.
\]
Gronwall proves~\eqref{eq:semilinear-bound}, and~\eqref{eq:rho-properties} proves~\eqref{eq:semilinear-psi}. For the projected equation, the feedback path takes values in $H_N$, so Theorem~\ref{thm:measure} uses only its occupied spectral bands. Proposition~\ref{prop:wellposed} supplies the dimension-independent constants.
\end{proof}

\begin{corollary}[Weighted reference paths]\label{cor:weighted-semilinear}
Under the hypotheses of Theorem~\ref{thm:semilinear}, suppose in addition that $B(v)\in C([-r,T];H_w)$ and $\norm{B(v)}_{C(H_w)}\leq R_w$. Then
\begin{equation}\label{eq:weighted-semilinear}
 \norm{u_\mu^\phi-u_\nu^\psi}_{Y_T}
 \leq C\bigl(\norm{\phi-\psi}_X+R_w\rhoAw(d_{\BL}(\mu,\nu))\bigr).
\end{equation}
The constant uses the same $H$-ball Lipschitz bounds as before, as well as $D_w$.
\end{corollary}
\begin{proof}
Use the weight $w$ in Theorem~\ref{thm:measure} when estimating~\eqref{eq:residual}. The remaining terms in~\eqref{eq:diff-mild} are still estimated in $H$.
\end{proof}

For nonlinear equations, weighted regularity of the entire feedback path is an additional hypothesis. It does not follow solely from weighted regularity of the initial history. The linear identity-feedback case is covered without that extra propagation assumption by Proposition~\ref{prop:linear}.

\section{Spatial resolution and memory quadrature}\label{sec:discrete}
Let
\[
 H_N=\spanof\{e_1,\ldots,e_N\},\qquad
 \Pi_N:H\to H_N,\qquad A_N=A|_{H_N}.
\]
The notation $\Pi_N$ distinguishes the Galerkin projection from the band projections $P_k$ in~\eqref{eq:bands}. Put
\begin{equation}\label{eq:bandcount}
 n_N=\#\calJ(A_N).
\end{equation}
The reference model in the sharp lower results of this section is the linear equation $f=0$, $B=\mathrm{Id}$, $h=0$. Its projected current-state map is denoted by $U_{\mu,N}(t)$.

\subsection{The optimal finite-resolution Lipschitz constant}
Fix $0<T<r/2$. Define
\begin{equation}\label{eq:L-N}
 L_{N,w}(T)=
 \sup_{\substack{\mu,\nu\in\calP(I),\ \mu\ne\nu\\
                 \phi\in C(I;H_N),\ \norm\phi_{X_w}\leq1}}
 \frac{\norm{U_{\mu,N}(T)\phi-U_{\nu,N}(T)\phi}_H}
      {d_{\BL}(\mu,\nu)}.
\end{equation}

\begin{theorem}[Sharp finite-resolution Lipschitz constants]\label{thm:L-N}
For every admissible weight,
\begin{equation}\label{eq:L-N-exact}
 L_{N,w}(T)\asymp
       \left(\sum_{k\in\calJ(A_N)}w_k^{-2}\right)^{1/2}.
\end{equation}
The comparison constants are independent of $N$. In particular,
\begin{equation}\label{eq:L-N-unweighted}
 L_{N,1}(T)\asymp\sqrt{n_N}.
\end{equation}
If $A$ is unbounded with compact inverse, $L_{N,1}(T)\to\infty$.
\end{theorem}
\begin{proof}
The weighted linear upper estimate gives a constant times $\rho_{A_N,w}(d)$ for $0<d\leq1$. Formula~\eqref{eq:rho} implies
\[
 \rho_{A_N,w}(d)/d
 \leq\left(\sum_{k\in\calJ(A_N)}w_k^{-2}\right)^{1/2}.
\]
For $d>1$, boundedness of the solution map by~\eqref{eq:linear-weighted} gives the same upper estimate after increasing the constant, since the sum includes the term $k=0$ and is at least one. For the lower estimate, take the two shifted Dirac kernels from Theorem~\ref{thm:lower} and let $a\downarrow0$. The quotient of their bounded-Lipschitz distance by $a$ tends to one, and
\[
 \lim_{a\downarrow0}\frac{\rho_{A_N,w}(a)}{a}
 =\left(\sum_{k\in\calJ(A_N)}w_k^{-2}\right)^{1/2}.
\]
Theorem~\ref{thm:lower} is uniform in $N$. Finally, an unbounded spectrum meets infinitely many dyadic bands, and the nested spaces eventually contain an eigenvector from each such band. Thus $n_N\to\infty$.
\end{proof}

\begin{corollary}[Two-resolution law for elliptic spectra]\label{cor:phase}
Suppose $A$ satisfies~\eqref{eq:gaps}. For $0<a\leq1$ and every $N$,
\begin{equation}\label{eq:phase-law}
 \rhoN(a)\asymp
 a\sqrt{1+\log\!\bigl(\min\{\lambda_N,a^{-1}\}\bigr)},
\end{equation}
and
\begin{equation}\label{eq:L-log}
 L_{N,1}(T)\asymp\sqrt{1+\log\lambda_N}.
\end{equation}
For the normalized Dirichlet Laplacian, the last expression is also comparable to $\sqrt{\log(e+N)}$.
\end{corollary}
\begin{proof}
Set $m=\lfloor\log_2(\min\{\lambda_N,a^{-1}\})\rfloor$. The occupied bands below $m$ each contribute $a^2$ in~\eqref{eq:rho}; there are at most $m+1$ of them and, by~\eqref{eq:gaps}, at least a fixed positive fraction of $m+1$, up to a fixed number of low bands. The band $0$ handles that fixed number. If $a^{-1}<\lambda_N$, the bands above $\lfloor\log_2(a^{-1})\rfloor$ contribute a geometric tail bounded by $Ca^2$. If $\lambda_N\leq a^{-1}$, all retained bands are already included, apart from an inessential endpoint convention for the last band. This proves~\eqref{eq:phase-law}. The same count gives $n_N\asymp1+\log\lambda_N$, and~\eqref{eq:L-log} follows from Theorem~\ref{thm:L-N}. Equation~\eqref{eq:eigen-growth} proves the final statement.
\end{proof}

For $a\lambda_N\leq1$, the law~\eqref{eq:phase-law} is $a\sqrt{1+\log\lambda_N}$. For $a\lambda_N\geq1$, it is the continuum modulus $a\sqrt{\log(e/a)}$. Both regimes refer to perturbing the memory law; the spatial approximation error is a separate quantity.

\begin{example}[Sparse spectra]\label{ex:sparse}
On $\ell^2$, let $\lambda_1=1$ and $\lambda_{j+2}=2^{2^j}$ for $j\geq0$. Then the occupied bands are $0,1,2,4,8,\ldots$, and~\eqref{eq:rho} gives
\begin{equation}\label{eq:sparse-modulus}
 \rhoA(a)\asymp a\sqrt{\log(e+\log(1/a))},\qquad 0<a\leq1.
\end{equation}
To see this, count the occupied indices below $\log_2(a^{-1})$; their number is comparable to $\log(e+\log(1/a))$, and the remaining geometric tail is $O(a^2)$ after squaring. Each retained eigenvalue occupies its own band, so $L_{N,1}(T)\asymp\sqrt N$. In contrast, $\sqrt{\log\lambda_N}$ is exponentially large in $N$. Thus a formula using only the spectral radius is not valid for arbitrary compact-resolvent generators.
\end{example}

\subsection{Atomic quadrature of the memory measure}
We consider only replacement of the delay measure at fixed spatial dynamics. Collocation and product-integration methods for Volterra and functional differential equations are developed in~\cite{brunner2004}; spectral approximation of delay generators and evolution operators for stability analysis is treated in~\cite{breda2015}. Operator-valued Volterra memory is a broader evolution framework~\cite{pruss1993}, distinct from the compact-interval finite-measure feedback here. The estimate below quantifies this specific kernel-replacement step, rather than the full discretization error of those methods.

Let $(I_j)_{j=1}^m$ be a Borel partition of $I$, with $\operatorname{diam}(I_j)\leq\Delta$, and choose $\theta_j\in I_j$. Define
\begin{equation}\label{eq:quad}
 Q_\Delta\mu=\sum_{j=1}^m\mu(I_j)\delta_{\theta_j}.
\end{equation}
Then
\begin{equation}\label{eq:quad-BL}
 \norm{Q_\Delta\mu}_{\TV}\leq\norm\mu_{\TV},
 \qquad d_{\BL}(Q_\Delta\mu,\mu)\leq\Delta\norm\mu_{\TV}.
\end{equation}
The first inequality follows from $\sum_j|\mu(I_j)|\leq|\mu|(I)$. For the second, write the difference against a Lipschitz test function as
\[
 \sum_j\int_{I_j}[q(\theta_j)-q(\theta)]\dd\mu(\theta).
\]
The measure error in~\eqref{eq:quad-BL} is elementary; the spectral response specifies its sharp propagated solution error.

\begin{corollary}[Quadrature upper bounds]\label{cor:quad-upper}
Under Assumption~\ref{ass:reaction}, for fixed $R,M,T$ and $0<\Delta\leq1$,
\begin{equation}\label{eq:quad-upper}
 \sup_{\substack{\norm\mu_{\TV}\leq M\\\norm\phi_X\leq R}}
 \norm{u_\mu^\phi-u_{Q_\Delta\mu}^\phi}_{Y_T}
 \leq C\rhoA(\Delta).
\end{equation}
For the $N$-mode equation the corresponding upper bound is $C\rhoN(\Delta)$, with a constant independent of $N$. For the linear equation on the unit ball of $X_w$, the bounds are $C\rhoAw(\Delta)$ and $C\rho_{A_N,w}(\Delta)$, respectively.
\end{corollary}
\begin{proof}
Apply Theorem~\ref{thm:semilinear},~\eqref{eq:quad-BL}, and the scaling inequality~\eqref{eq:rho-scale}. The weighted assertion follows from~\eqref{eq:linear-weighted} and Corollary~\ref{cor:weighted-semilinear}.
\end{proof}

The following theorem specifies the quadrature rule for which the lower bound is asserted. Let $\Delta=r/m$ and let $Q_\Delta$ be midpoint quadrature on the uniform partition into $m$ cells, with any fixed assignment of cell endpoints. The points used in the proof lie strictly inside their cells, so this assignment does not matter.

\begin{theorem}[Sharp fixed-grid memory quadrature]\label{thm:quad-lower}
Fix $0<T<r/4$. For the linear equation and all sufficiently large $m$,
\begin{equation}\label{eq:quad-sharp}
 \sup_{\substack{\mu\in\calP(I)\\\norm\phi_{X_w}\leq1}}
 \norm{U_\mu(T)\phi-U_{Q_\Delta\mu}(T)\phi}_H
 \asymp\rhoAw(\Delta).
\end{equation}
For the $N$-mode equation the analogous quantity is comparable to $\rho_{A_N,w}(\Delta)$, uniformly in $N$. Consequently, for the Dirichlet Laplacian and $w=1$, the continuum rate is
\begin{equation}\label{eq:quad-sharp-elliptic}
 \Delta\sqrt{\log(e/\Delta)},
\end{equation}
and the joint displacement--resolution rate is obtained by replacing $a$ by $\Delta$ in~\eqref{eq:phase-law}.
\end{theorem}
\begin{proof}
The upper estimate is Corollary~\ref{cor:quad-upper}. Choose a cell whose midpoint $\theta_j$ tends to $-r/2$ as $m\to\infty$, and place a unit atom at $\theta_j-\Delta/4$. Midpoint quadrature moves this atom to $\theta_j$. The two delays are therefore $\tau=-\theta_j$ and $\tau+a$, with $a=\Delta/4$. For large $m$ these lags remain in a fixed compact subinterval of $(T,r)$, and~\eqref{eq:time-choice} holds uniformly. Theorem~\ref{thm:lower} and~\eqref{eq:rho-scale} give the lower bound by $c\rhoAw(\Delta)$. The same argument uses only the retained spectral bands for $A_N$.
\end{proof}

\begin{remark}[Meaning of quadrature optimality]\label{rem:quad-scope}
Theorem~\ref{thm:quad-lower} concerns worst-case measures and histories for this prescribed grid rule. It is not a minimax lower bound over all adaptive atomic representations, and it does not exclude higher-order estimates for one fixed smooth kernel. An adaptive rule that retains an input atom at its original location has zero error for that atom. The lower-bound measure here is chosen to be moved by the prescribed rule.
\end{remark}

\subsection{A joint space--memory approximation estimate}
In this subsection let $A$ have infinitely many eigenvalues. To compare the continuum equation to its spatial projection, let $D\subset X$ be bounded, and set
\begin{equation}\label{eq:projection-defect}
 \varepsilon_N(D)=\sup_{\phi\in D}
           \norm{(I-\Pi_N)\phi}_{C(I;H)}.
\end{equation}
For $\phi\in D$, let $u_{\nu,N}^{\Pi_N\phi}$ solve
\begin{equation}\label{eq:galerkin}
 \partial_t u_N+A_Nu_N+\Pi_N f(u_N)
 =\int_I\Pi_N B(u_N(t+\theta))\dd\nu(\theta)+\Pi_Nh,
 \qquad u_N|_I=\Pi_N\phi.
\end{equation}
Here $\Pi_Nh$ is interpreted through the $V',V$ pairing. Identify $u_N$ with its embedding into $H$.

\begin{theorem}[Separate spatial and memory errors]\label{thm:joint}
Under Assumption~\ref{ass:reaction}, put $z=A^{-1}h$ and $\Lambda_N=\lambda_{N+1}$. For $\norm\mu_{\TV},\norm\nu_{\TV}\leq M$,
\begin{align}
 \sup_{\phi\in D}\norm{u_\mu^\phi-u_{\nu,N}^{\Pi_N\phi}}_{Y_T}
 \leq C\bigl[&\varepsilon_N(D)+\norm{(I-\Pi_N)z}_H
       +\Lambda_N^{-1}\notag\\
       &+\rhoN(d_{\BL}(\mu,\nu))\bigr],\label{eq:joint}
\end{align}
where $C$ is independent of $N,\mu,\nu$. Thus, for $\nu=Q_\Delta\mu$, the final term can be replaced by $C\rhoN(\Delta)$.
\end{theorem}
\begin{proof}
First compare the continuum and projected equations with the same measure $\mu$. Write $Q_N=I-\Pi_N$. By Proposition~\ref{prop:wellposed}, the paths are in a common $H$-ball, and
\[
 F_\mu(s)=-f(u_\mu^\phi(s))+D_\mu u_\mu^\phi(s)
\]
is bounded in $H$ by a common constant $R_F$. Formula~\eqref{eq:shifted-mild}, the commutation of $Q_N$ with $E(t)$, and the spectral tail estimate give, for $t\geq0$,
\begin{align*}
 \norm{Q_Nu_\mu^\phi(t)}
 &\leq\norm{Q_N\phi(0)}+\norm{Q_Nz}
       +R_F\int_0^t e^{-\Lambda_N(t-s)}\dd s\\
 &\leq\varepsilon_N(D)+\norm{Q_Nz}+R_F\Lambda_N^{-1}.
\end{align*}
The same bound holds on negative times, by~\eqref{eq:projection-defect}. Let $p=\Pi_Nu_\mu^\phi-u_{\mu,N}^{\Pi_N\phi}$. It has zero initial history, and subtraction of the projected mild formulas yields
\[
 \norm{p(t)}\leq C\int_0^t
   \sup_{-r\leq s\leq\sigma}
       \bigl(\norm{Q_Nu_\mu^\phi(s)}+\norm{p(s)}\bigr)\dd\sigma.
\]
Gronwall bounds the full spatial error by the first three terms in~\eqref{eq:joint}. Next compare $u_{\mu,N}^{\Pi_N\phi}$ and $u_{\nu,N}^{\Pi_N\phi}$ using Theorem~\ref{thm:semilinear} on $H_N$. This gives the last term. A triangle inequality proves~\eqref{eq:joint}, and~\eqref{eq:quad-BL} proves its quadrature version.
\end{proof}

If $D$ is compact in $X$, then $\varepsilon_N(D)\to0$. One way to see this is to note that $\{\phi(\theta):\phi\in D,\theta\in I\}$ is compact in $H$ and strong convergence of the uniformly bounded projections is uniform on that compact set. If instead $D$ has a uniform $D(A^\beta)$ bound for some $\beta>0$, then
\[
 \varepsilon_N(D)\leq\Lambda_N^{-\beta}
                \sup_{\phi\in D}\norm{A^\beta\phi}_{C(I;H)}.
\]
Also $\norm{Q_Nz}\leq\Lambda_N^{-1/2}\norm h_{V'}$. These facts turn~\eqref{eq:joint} into concrete convergence bounds without imposing a false approximation property on a rough unit ball. Indeed, on the unit ball of $X$, $\varepsilon_N(D)=1$ for every $N$, as witnessed by constant histories equal to $e_{N+1}$. Uniform convergence in $Y_T$, which includes the initial history, cannot hold there. The kernel-stability obstruction and this spatial-projection obstruction are distinct.

\begin{remark}[Uniform approximation of positive-time states]\label{rem:positive-time}
The preceding failure of uniform history-space convergence is not a failure of uniform current-state approximation. For the commuting linear model~\eqref{eq:linear-model}, keep the full time history and truncate only the spatial modes. Uniqueness gives $U_{\mu,N}(T)\Pi_N\phi=\Pi_NU_\mu(T)\phi$. For $\norm\mu_{\TV}\leq M$, $\norm\phi_X\leq1$, and every fixed $T>0$, the mild formula and~\eqref{eq:linear-weighted} therefore give
\begin{equation}\label{eq:positive-time}
 \norm{U_\mu(T)\phi-U_{\mu,N}(T)\Pi_N\phi}_H
 \leq e^{-\lambda_{N+1}T}
       +\frac{Me^{MT}}{\lambda_{N+1}}.
\end{equation}
Indeed, apply $I-\Pi_N$ to the mild formula, use its commutation with $E(t)$, and integrate $e^{-\lambda_{N+1}(T-s)}$ against the bound $Me^{MT}$ for the delayed forcing. The right side tends to zero uniformly over these measures and histories. This is compatible with the growing kernel-Lipschitz constants in Theorem~\ref{thm:L-N}: approximation accuracy at a fixed positive time and uniform Lipschitz sensitivity to the memory law are different properties.
\end{remark}

\begin{remark}[Accuracy and uniform Lipschitz stability]\label{rem:surrogate}
The positive-time convergence in Remark~\ref{rem:positive-time} can be compared with the sharp delay-law modulus to determine whether uniform accuracy is compatible with a common Lipschitz bound. Fix $0<T<r/2$. Suppose an approximation $\widehat F(\mu,\phi)$ to $U_\mu(T)\phi$ has uniform error $\epsilon$ in $H$ on $\mathcal P(I)\times B_X$, where $B_X$ is the unit ball of $X$. Suppose also that it is $L$-Lipschitz in $d_{\BL}$ uniformly over $\phi\in B_X$. The triangle inequality and Corollary~\ref{cor:omega} then give, for small $d$,
\[
 2\epsilon\geq\Omega_{A,1,T}(d)-Ld\geq c\rhoA(d)-Ld.
\]
For unbounded $A$, the occupied-band formula~\eqref{eq:rho} gives $\rhoA(d)/d\to\infty$ as $d\downarrow0$. Thus no sequence of approximations with a common finite bound on these Lipschitz constants can converge uniformly on $\mathcal P(I)\times B_X$. The Galerkin approximations in Remark~\ref{rem:positive-time} converge uniformly while their Lipschitz constants grow as described in Theorem~\ref{thm:L-N}.

This comparison is a consequence of the target modulus and depends on the specified history class. For elliptic spectra, Theorem~\ref{thm:log-hierarchy} gives uniform Lipschitz stability on bounded balls of $C(I;D((1+\log A)^\gamma))$ when $\gamma>1/2$, so the obstruction from the rough history ball does not persist on these weighted classes.
\end{remark}

\section{Conclusion}\label{sec:conclusion}
For positive self-adjoint generators, the weighted occupied-band sum~\eqref{eq:rho} characterizes the norm of an integrated delay perturbation and the worst-case response of a linear memory equation. The exact identity~\eqref{eq:exact-realization} transfers the spectral lower bound to positive point delays without endpoint errors. Elliptic spectra yield the endpoint modulus $d\sqrt{\log(e/d)}$; sparse spectra can give different rates.

The sharpness statement is inseparable from its data class. For the commuting linear model, the reciprocal-square criterion~\eqref{eq:weighted-L-criterion} gives precisely when a spatial weight restores Lipschitz stability. Elliptic spectral density places the logarithmic threshold at $\gamma=1/2$. Thus the result describes how weak spatial regularity changes an endpoint modulus, rather than a general obstruction to approximating evolution operators.

Sharp spectral-truncation constants and prescribed midpoint-quadrature errors follow from the same formula. The joint estimate separates spatial and memory errors;~\eqref{eq:positive-time} shows that positive-time approximation can coexist with increasing memory sensitivity. The semilinear result supplies upper bounds for $H$-valued reactions and feedbacks that are Lipschitz on bounded balls. The sharp analysis does not cover pointwise cubic reactions on $L^2$, nonnormal generators, unbounded delayed feedback, state-dependent delays, or infinite memory.

\backmatter
\section*{Statements and declarations}
\subsection*{Funding}
This work was conducted independently. No funding was received for this study.
\subsection*{Competing interests}
The author has no relevant financial or non-financial interests to disclose.
\subsection*{Data availability}
This is a theoretical study. All results are established by the analytical proofs in the article; no external datasets or computational results are required to reproduce the stated conclusions.

\bibliography{references}
\end{document}